\documentclass[a4paper,10pt]{amsart}
\usepackage[utf8x]{inputenc}

\usepackage{amssymb}
\usepackage[all,cmtip]{xy}
\usepackage[dvips]{graphicx}
\usepackage{graphics}
\usepackage{calrsfs}
\usepackage{hyperref}
\usepackage{xcolor}

\newtheorem{theorem}{Theorem}[section]
\newtheorem{lemma}[theorem]{Lemma}
\newtheorem{prop}[theorem]{Proposition}

\newtheorem{crlr}[theorem]{Corollary}

\theoremstyle{definition}
\newtheorem{rem}[theorem]{Remark}

\newcommand{\mf}[1]{\mathfrak{#1}}
\newcommand{\mc}[1]{\mathcal{#1}}

\numberwithin{equation}{section}

\title{The weak commutativity construction for Lie algebras}
 \author{Luis Augusto de Mendon\c ca}
 \address{Department of Mathematics, University of Campinas (UNICAMP), rua S\'{e}rgio Buarque de Holanda, 651, 13083-859, Campinas-SP, Brazil.}
 \email{luismendonca@ime.unicamp.br}
 \keywords{Lie Algebras, homological finiteness properties, finite presentability, Schur multiplier}
 \subjclass[2010]{17B55, 20F05, 20J05}
\begin{document}

 \begin{abstract}
 We study the analogue of Sidki's weak commutativity construction, defined originally for groups, in the category of Lie algebras. This is the
 quotient $\chi(\mf{g})$ of the Lie algebra freely generated by two isomorphic copies $\mf{g}$ and $\mf{g}^{\psi}$ of a fixed Lie algebra by the ideal generated by
 the brackets $[x,x^{\psi}]$, for all $x$. We exhibit an abelian ideal of 
 $\chi(\mf{g})$ whose associated quotient is a subdirect sum in $\mf{g} \oplus \mf{g} \oplus \mf{g}$ and we give conditions for this ideal to be
 finite dimensional. We show that $\chi(\mf{g})$ has a subquotient that is isomorphic to the Schur multiplier of $\mf{g}$. We prove that $\chi(\mf{g})$
 is finitely presentable or of homological type $FP_2$ if and only if $\mf{g}$ has the same property, but $\chi(\mf{f})$ is not of type $FP_3$ if $\mf{f}$ is a 
 non-abelian free Lie algebra.
 \end{abstract}

\maketitle

\begin{section}{Introduction}
 The weak commutativity construction was first defined for groups by Sidki \cite{Sidki} and goes as follows: for a group $G$, we define $\chi(G)$
 as the quotient of the free product $G \ast G^{\psi}$ of two isomorphic copies of $G$ by the normal subgroup generated by the elements $[g,g^{\psi}]$
 for all $g \in G$. We think of this as a functor that receives the group $G$ and returns the group with weak commutativity $\chi(G)$.
 
 In a series of papers, many group theoretic properties were shown to be preserved by this functor. For instance, it preserves finiteness and
 solvability \cite{Sidki} and finite presentability \cite{BridsonKoch}. Moreover, if $G$ is finitely generated nilpotent, polycyclic-by-finite or
 solvable of type $FP_{\infty}$, then
 $\chi(G)$ has the same property \cite{GuRoSi, LimaOliveira, KochSidki}.
 
 The group $\chi(G)$ has a chain of normal subgroups with some nice properties, which allows some of the proofs of the results cited above to
 be carried on. We write this series as $R(G) \subseteq W(G) \subseteq L(G)$ (or $R \subseteq W \subseteq L$ if
 $G$ is understood) and we observe the following: $W$ is always abelian and  
 $\chi(G)/W$ is isomorphic to a subdirect product living inside $G \times G \times G$; the subquotient $W/R$ is isomorphic to the Schur multiplier of $G$ 
 \cite{Rocco}; and $\chi(G)$ is a split extension of $L$ by $G$. 
 
 We consider in this paper an analogue of this construction in the category of Lie algebras over a field. We fix once and for all a field $K$
 with $char(K) \neq 2$, and we only consider Lie algebras over $K$.
 For any Lie algebra $\mf{g}$, let $\mf{g}^{\psi}$ be an isomorphic copy, with isomorphism written as $x \mapsto x^{\psi}$. We define
 \[\chi(\mf{g}) = \langle \mf{g}, \mf{g}^{\psi} \hbox{ } | \hbox{ } [x,x^{\psi}]=0 \hbox{ for all } x \in \mf{g} \rangle.\]
 We show that $\chi(\mf{g})$ has a chain of ideals
 \[ R(\mf{g}) \subseteq W(\mf{g}) \subseteq L(\mf{g}) \subseteq \chi(\mf{g})\]
 satisfying the analogous properties as the chain of normal subgroups in the group case. Again, we write only $R$, $W$ and $L$ if there is no risk of confusion. So $W$ is an abelian ideal,
 $\chi(\mf{g})/W$ is a subdirect sum living inside $\mf{g} \oplus \mf{g} \oplus \mf{g}$, the subquotient 
 $W/R$ is isomorphic to $H_2(\mf{g}; K)$ and $\chi(\mf{g})$ is a split extension of $L$ by $\mf{g}$.
 
 We denote by $\mc{U}(\mf{g})$ the universal enveloping algebra of $\mf{g}$. Recall that $\mf{g}$ is of homological type $FP_m$ if the trivial 
 $\mc{U}(\mf{g})$-module $K$ admits a projective resolution
 \[\mathcal{P}: \ldots \to P_n \to P_{n-1} \to \ldots \to P_1 \to P_0 \to K \to 0\]
 with $P_j$ finitely generated for all $j \leq m$. We say that $\mf{g}$ is of type $FP_{\infty}$ if it is of type $FP_m$ for all $m \geq 0$. A
 Lie algebra is finitely generated if and only if it is of type $FP_1$, and it is of type $FP_2$ if it is finitely presentable.  
  
 There are few results about finite presentability of Lie algebras. For metabelian and center-by-metabelian Lie algebras, some methods developed
 by Bryant and Groves in \cite{BG1, BG2, BG3} give a nice picture. Less is known about homological finiteness properties, but Groves and Kochloukova 
 showed in \cite{GrovesKoch} that solvable Lie algebras of type $FP_{\infty}$ are finite dimensional. This is inspired by a result by 
 Kropholler that says that solvable groups of type $FP_{\infty}$ are constructible \cite{Kropholler}.

 Our first result relies on some work of Kochloukova and Martínez-Pérez on subdirect sums of Lie algebras \cite{KochMart}.
 \begin{theorem} \label{T1}
  If $\mf{g}$ is of type $FP_2$ or finitely presentable, then $\chi(\mf{g})/W$ has the same property.
 \end{theorem}
 
 Again, apart from the aforementioned paper, little is known about finiteness properties of subdirect sums of Lie algebras. In the group-theoretic
 case there are several results. For finite presentability and related homotopical finiteness properties of subdirect products of groups, see 
 \cite{BR, BHMS1, BHMS2, BHMS3, Mart}
 to cite a few. Some homological counterparts of the results in these papers were treated in \cite{Kuckuck} and \cite{Francismar}.
 
 Theorem \ref{T1}, together with the exactness of $W \rightarrowtail \chi(\mf{g}) \twoheadrightarrow \chi(\mf{g})/W$, could be used to deduce
 finiteness properties for $\chi(\mf{g})$ when $W$ is finite dimensional. In the same spirit of Theorem B 
 in \cite{KochSidki}, we give a sufficient condition for that. 
 
 \begin{theorem}  \label{T2}
If $\mf{g}$ is of type $FP_2$ and $\mf{g}'/\mf{g}''$  is finite dimensional, then $W(\mf{g})$ is finite dimensional.
 \end{theorem}

  In \cite{KochSidki} the authors show that the group-theoretic weak commutativity construction preserves the property of being solvable of type
  $FP_{\infty}$. The analogous result also holds for Lie algebras, but for a much simpler reason. By Theorem 1 in \cite{GrovesKoch}, if $\mf{g}$ 
  is solvable of type $FP_{\infty}$, then it is finite dimensional. In this case of course $\mf{g}'/\mf{g}''$ is also finite dimensional, and then
  so is $W$, by the theorem above. Moreover, $\chi(\mf{g})/W$ is clearly finite dimensional and solvable, being a Lie subalgebra of 
  $\mf{g} \oplus \mf{g} \oplus \mf{g}$. Thus we have the following corollary.
 
 \begin{crlr}
  If $\mf{g}$ is solvable of type $FP_{\infty}$, then so is $\chi(\mf{g})$.
 \end{crlr}

 The same reasoning above can be used if we assume at first that $\mf{g}$ is finite dimensional.
 
 \begin{crlr}
 If $\mf{g}$ is finite dimensional, then so is $\chi(\mf{g})$.
 \end{crlr}

 The condition $\mf{g}'/\mf{g}''$ \textit{is finite dimensional} is strong and does not apply, for instance, to free non-abelian Lie algebras. We 
 do not have, however, any example of a finitely presentable Lie algebra $\mf{g}$ such that $W(\mf{g})$ is not finite dimensional. The search
 for such an example is complicated because we do not have a method to determine that an element of $W(\mf{g})$ is non-zero. 
 
 The following theorem shows that we actually do not need the restriction that $\mf{g}'/\mf{g}''$ is finite dimensional to study finite presentability.
 
 \begin{theorem}  \label{T3}
 Let $\mf{g}$ be a Lie algebra. Then $\mf{g}$ is finitely presentable (resp. of type $FP_2$) if and only if $\chi(\mf{g})$ is finitely 
 presentable (resp. of type $FP_2$). If $\mf{f}$ is a free non-abelian Lie algebra, then $\chi(\mf{f})$ is not of type $FP_3$.
 \end{theorem}

In proof of the theorem we implicitly use the construction of HNN-extensions for Lie algebras, as defined by Wasserman in \cite{Wasserman}.

 Finally, we have the analogous result on the Schur multiplier as in the group case.
 
 \begin{theorem} \label{T4}
 For all $\mf{g}$ we have $W(\mf{g})/R(\mf{g}) \simeq H_2(\mf{g}; K)$ as vector spaces over $K$.
 \end{theorem}

 The group-theoretic counterpart of this is contained in \cite{Sidki} and \cite{Rocco}. To prove our version, we use the description of the Schur multiplier of Lie 
 algebras given by Ellis \cite{Ellis}. 
 
 We observe that from the theorem above we get that if $\mf{g}$ is finitely presentable, then $W(\mf{g})/R(\mf{g})$ is finite dimensional. Thus the question of whether
 there is a finitely presentable Lie algebra $\mf{g}$ such that $W(\mf{g})$ is infinite dimensional reduces to the same question 
 with respect to $R(\mf{g})$.
 
 We have analyzed $R(\mf{g})$ for some specific Lie algebras. It turns out that $R(\mf{g})$ is trivial whenever $\mf{g}$ is abelian, in 
 contrast with the case of groups: $R(G) \neq 1$ when $G$ is any elementary abelian $2$-group of order at least $8$ (\cite{LimaOliveira}, 
 Proposition 4.5). We also showed that $R(\mf{g})$ is zero if $\mf{g}$ is perfect or $2$-generated. 
 
 In any of these cases (that is, when $R(\mf{g})=0$) and for $m\geq 2$, we have that $\chi(\mf{g})$ is of type $FP_m$ if and only if $\chi(\mf{g})/W(\mf{g})$ is of type $FP_m$. This is because under each of these hypotheses $\mf{g}$ is of type $FP_2$ (being a retract of both $\chi(\mf{g})$ and $\chi(\mf{g})/W(\mf{g})$), thus 
 the Schur multiplier $H_2(\mf{g};K)$ is finite dimensional. This is especially interesting because $\chi(\mf{g})/W(\mf{g})$ has a more concrete description 
 as a subdirect sum living inside $\mf{g} \oplus \mf{g} \oplus \mf{g}$.  
 
 We used GAP \cite{GAP4} to obtain examples of Lie algebras over $\mathbb{Q}$ with $R \neq 0$. For this we considered $3$-generated nilpotent
 Lie algebras. From that also follows that $R(\mf{f}) \neq 0$ for $\mf{f}$ free of rank at least $3$. 
 \end{section}

\begin{section}{Preliminaries on homological properties of Lie algebras}  \label{sec2}
We review here the definitions and some standard results on Lie algebras that we will use throughout the paper.

Recall that we have fixed a field $K$ with $char(K) \neq 2$, and we only consider Lie algebras defined over this field. For 
any Lie algebra $\mf{g}$ and $S \subseteq \mf{g}$ a subset, we denote by $\langle S \rangle$ and by $\langle \langle S \rangle \rangle$
the subalgebra and the ideal, respectively, generated by $S$. 

We denote by $\mc{U}(\mf{g})$ the universal enveloping algebra of $\mf{g}$. It can be described as the quotient of the tensor algebra on the 
vector space $\mf{g}$ by the ideal generated by the elements $xy-yx-[x,y]$ for all $x,y \in \mf{g}$. The obvious map 
$i: \mf{g} \to \mc{U}(\mf{g})$ is injective by the Poincaré-Birkhoff-Witt theorem. The augmentation ideal, which we denote by 
$Aug(\mc{U}(\mf{g}))$, is the kernel of the homomorphism 
$\mc{U}(\mf{g}) \to K$ taking $i(\mf{g})$ to $0$. 

For a $\mc{U}(\mf{g})$-module $A$, the $n$-th homology of $\mf{g}$ with coefficients in $A$ is defined as
\[ H_n(\mf{g}; A) := Tor_n^{\mc{U}(\mf{g})}(A,K).\]
If $K$ is the trivial $\mc{U}(\mf{g})$-module, then $H_1(\mf{g}; K) \simeq \mf{g}/\mf{g}'$. 
The second homology can be computed by a Hopf formula analogue.
\begin{lemma}[Hopf formula]
Let $\mf{f}$ be a free Lie algebra and let $\mf{r} \subseteq \mf{f}$ be an ideal. Denote $\mf{g} = \mf{f}/\mf{r}$. Then
\[H_2(\mf{g}; K) \simeq \frac{[\mf{f},\mf{f}] \cap \mf{r}}{[\mf{f},\mf{r}]}.\]
\end{lemma}

The following general facts about cohomology of Lie algebras can be found in \cite{Weibel}, Chapter 7. For a short exact sequence 
$\mf{h} \rightarrowtail \mf{g} \twoheadrightarrow \mf{q}$ of Lie algebras and a $\mc{U}(\mf{g})$-module $A$, there is an associated Lyndon-Hochschild-Serre spectral sequence
\[ E_{p,q}^2= H_p(\mf{q}; H_q(\mf{h}; A)) \Rightarrow H_{p+q}(\mf{g}; A),\]
which is convergent and concentrated in the first quadrant. The differential $d^r$ of the $r$-th page has bidegree $(-r,r-1)$. If $A$ is a trivial $\mc{U}(\mf{g})$-module,
then the associated $5$-term exact sequence is written as
\[H_2(\mf{g}; A) \to H_2(\mf{q}; A) \to H_0(\mf{q}; H_1(\mf{h}; A)) \to H_1(\mf{g}; A) \to H_1(\mf{q};A) \to 0.\]
There is also Künneth formula to compute the homology of direct sums:
\[H_n( \mf{g} \oplus \mf{h}; K) \simeq \oplus_{0 \leq i \leq n} H_i(\mf{g}; K) \otimes_K H_{n-i}(\mf{h}; K).\]

A $\mc{U}(\mf{g})$-module $A$ is said to be of type $FP_m$ if it admits a projective resolution
\[\mathcal{P}: \ldots \to P_n \to P_{n-1} \to \ldots \to P_1 \to P_0 \to A \to 0\]
with $P_j$ finitely generated for all $j \leq m$. Part of the relevance of modules of type $FP_m$ is explained by the following lemma.

\begin{lemma}
 Let $\mf{g}$ be a Lie algebra and let $A$ be a $\mc{U}(\mf{g})$-module. If $A$ is of type $FP_m$, then $H_j(\mf{g};A)$ is finite dimensional
 for all $j \leq m$.
\end{lemma}

\begin{proof}
Let 
\[ \mathcal{P}: \ldots \to P_n \to P_{n-1} \to \ldots \to P_1 \to P_0 \to A \to 0\]
be a projective resolution of $A$ with $P_j$ finitely generated for all $j \leq m$. Then $H_{\ast}(\mf{g}; A)$ is the 
homology of the complex $ \mc{P}^{del} \otimes_{\mc{U}(\mf{g})} K$, where $\mc{P}^{del}$ is the deleted resolution. In particular, $H_i(\mf{g}; A)$ is
a subquotient of $P_i \otimes_{\mc{U}(\mf{g})} K$, which is finite dimensional when $P_i$ is finitely generated as a $\mc{U}(\mf{g})$-module.
\end{proof}

We say simply that $\mf{g}$ is of type $FP_m$ if $K$ is of type $FP_m$ as a trivial $\mc{U}(\mf{g})$-module. Notice that by the previous lemma
if $\mf{g}$ is of type $FP_2$, then $H_2(\mf{g};K)$ is finite dimensional.

As we have already 
mentioned, $FP_1$ is the same as finite generation, and finite presentability 
implies the property $FP_2$. For these results see the survey \cite{Survey}. We also have the following.

\begin{lemma}  \label{fpFP2}
A Lie algebra $\mf{g}$ is of type $FP_2$ if and only if there is a finitely presented Lie algebra $\mf{h}$ and an ideal 
$\mf{r} \subseteq \mf{h}$ such that $\mf{g} \simeq \mf{h}/\mf{r}$ and 
$\mf{r}=\mf{r}'$.
\end{lemma}

This is well-known for groups. The proof that we can find in \cite{BridsonKoch}, Lemma 3.1, can be easily adapted to Lie algebras. Notice that 
this depends on the equivalence of the concepts of ``Lie algebra of type $FP_2$'' and ``almost finitely presented Lie algebra''. Again, the proof of Proposition 2.2 in \cite{Bieri} 
works perfectly fine when translated to Lie algebras.

It is clear that retracts of finitely presentable Lie algebras are again finitely presentable. This is also true for the homological finiteness 
properties that we consider. Again, this is well-known for groups. 

 \begin{lemma}  \label{retractsFPm}
 If $\mf{g}$ is a Lie algebra of type $FP_m$ and $\mf{b}$ is a retract of $\mf{g}$ (that is, $\mf{g} \simeq \mf{a} \rtimes \mf{b}$ 
 for some ideal $\mf{a}$), then $\mf{b}$ is also of type $FP_m$.
\end{lemma}
\begin{proof}
By Theorem 1.3 in \cite{Bieri}, we have that $\mf{g}$ is of type $FP_m$ if and only if it is finitely generated and $H_j(\mf{g}; \prod_{\lambda \in \Lambda} \mathcal{U}(\mf{g}))=0$ for all
$1 \leq j < m$ and for all index sets $\Lambda$. 

Clearly if $\mf{g}$ is finitely generated, then so is $\mf{b}$, so we only need to worry about the properties of the homologies. We argue
as in \cite{Bux}, Proposition 4.1. If $\pi: \mf{g} \to \mf{b}$ and $\sigma: \mf{b} \to \mf{g}$ are the split homomorphisms, then for 
any index set $\Lambda$ there are induced homomorphisms on homology 
\[\pi_{\ast}: H_j(\mf{g}; \prod_{\lambda \in \Lambda} \mathcal{U}(\mf{g})) \to H_j(\mf{b}; \prod_{\lambda \in \Lambda} \mathcal{U}(\mf{b}))\]
and
\[\sigma_{\ast}: H_j(\mf{b}; \prod_{\lambda \in \Lambda} \mathcal{U}(\mf{b})) \to H_j(\mf{g}; \prod_{\lambda \in \Lambda} \mathcal{U}(\mf{g})),\]
for all $j$. It follows by functoriality that $\pi_{\ast} \circ \sigma_{\ast} = id$. Thus $\sigma_{\ast}$ is injective, from what follows that 
$H_j(\mf{b}; \prod_{\lambda \in \Lambda} \mathcal{U}(\mf{b}))=0$ for all $ 1 \leq j < m$ when $\mf{g}$ is of type $FP_m$, so in this case
$\mf{b}$ is also of type $FP_m$.
\end{proof}
 
Another fact that we will need is that if $\mf{g}$ is finite dimensional, then $\mc{U}(\mf{g})$ is left (and right) noetherian as 
a ring (Proposition 6 of I.2.6 \cite{Bourbaki}). In particular, in this case $\mf{g}$ is of type $FP_{\infty}$.

Finally, we will make use of a technical result on finite dimensional modules over free Lie algebras. This is Lemma 3.1 in \cite{KochMart}.

\begin{lemma}  \label{lemma3.1KM}
 Let $\mf{f}$ be a free Lie algebra and let $A$ be a $\mc{U}(\mf{f})$-module. Suppose that there is some
 $c \in \mc{U}(\mf{f}) \smallsetminus \{0\}$
 such that both $H_1(\mf{f}; A)$ and $c A$ are finite dimensional. Then $A$ is finite dimensional.
\end{lemma}

\end{section}

\begin{section}{The weak commutativity construction}
In this section we define and establish the first properties of $\chi(\mf{g})$ and its chain of ideals $R \subseteq W \subseteq L$. Let $\mf{g}$ be a Lie algebra and let $\mf{g}^{\psi}$ 
be an isomorphic copy. For any $x\in \mf{g}$, we denote by $x^{\psi}$ its image in 
$\mf{g}^{\psi}$. We define $\chi(\mf{g})$ by the presentation
\[ \chi(\mf{g}) = \langle \mf{g}, \mf{g}^{\psi} \hbox{ } | \hbox{ } [x,x^{\psi}] = 0 \hbox{ } \forall x \in \mf{g} \rangle.\]
This must be understood as the quotient of the free Lie sum of $\mf{g}$ and $\mf{g}^{\psi}$ by the ideal generated by the elements $[x,x^{\psi}]$,
for all $x \in \mf{g}$.

Let $L=L(\mf{g})$ be the ideal of $\chi(\mf{g})$ generated by the elements of the form $x-x^{\psi}$, for all $x \in \mf{g}$. Equivalently, $L$ is the kernel
of the homomorphism $\alpha: \chi(\mf{g}) \to \mf{g}$ defined by $\alpha(x) = \alpha(x^{\psi}) = x$ for all $x \in \mf{g}$. It is clear that this
homomorphism is split, that is, $\chi(\mf{g}) \simeq L \rtimes \mf{g}$.

\begin{lemma} \label{Lgenalg}
 $L$ is generated \textit{as a Lie algebra} by the elements $x-x^{\psi}$, for all $x \in \mf{g}$.
\end{lemma}

\begin{proof} 
 By the relations that define $\chi(\mf{g})$ we have that $[x+y, (x+y)^{\psi}]=0$ for all $x,y \in \mf{g}$, so
 \[ 0 = [x+y, (x+y)^{\psi}]= [x,x^{\psi}] + [x,y^{\psi}] + [y,x^{\psi}] + [y,y^{\psi}] = [x,y^{\psi}] + [y,x^{\psi}].\]
 Thus $[x,y^{\psi}] = [x^{\psi},y]$ for all $x,y \in \mf{g}$.

 Denote by $A$ the Lie subalgebra of $\chi(\mf{g})$ generated by the elements $y-y^{\psi}$, for all $y \in \mf{g}$. We want to show that 
 $A=L$, and for that it is enough to show that $[x,y-y^{\psi}] \in A$ and $[x^{\psi},y-y^{\psi}] \in A$ 
 for all $x,y \in \mf{g}$. Even more, as $[x,y-y^{\psi}] - [x^{\psi},y-y^{\psi}] = [x-x^{\psi},y-y^{\psi}] \in A$, it suffices to show that any of these
 commutators is an element of $A$.
 
 Now
 \[a_1 := [x-x^{\psi},y-y^{\psi}] = [x,y] - 2[x,y^{\psi}] + [x^{\psi},y^{\psi}] \in A\]
 and
 \[a_2 := [x,y]-[x,y]^{\psi} = [x,y] -[x^{\psi},y^{\psi}] \in A.\]
 But then
 \[2[x,y-y^{\psi}] = a_1 + a_2 \in A,\]
 which proves that $[x,y-y^{\psi}] \in A$ provided that $char(K) \neq 2$.
\end{proof}

Similarly, let $D=D(\mf{g})$ be the ideal of $\chi(\mf{g})$ generated by the elements of the form $[x,y^{\psi}]$ for all $x,y \in \mf{g}$. It can be seen
as the kernel of the homomorphism $\beta: \chi(\mf{g}) \to \mf{g} \oplus \mf{g}$ defined by $\beta(x) = (x,0)$ and $\beta(x^{\psi}) = (0,x)$ 
for all $x \in \mf{g}$.

By putting together $\alpha$ and $\beta$, we define a new homomorphism. Define
\[ \rho: \chi(\mf{g}) \to \mf{g} \oplus \mf{g} \oplus \mf{g}\]
by $\rho(x) = (x,x,0)$ and $\rho(x^{\psi})=(0,x,x)$ for all $x \in \mf{g}$, and let $W=W(\mf{g})$ be its kernel. Notice that
\[\rho(z) = (\beta_1(z), \alpha(z), \beta_2(z))\]
for all $z \in \chi(\mf{g})$, where $\beta_1$ and $\beta_2$ are the two components of $\beta$. Then clearly $ker(\rho) = ker(\alpha) \cap ker(\beta)$,
that is, $W = L \cap D$.

\begin{lemma}
For all $\mf{g}$, we have $[D,L]=0$.
\end{lemma}

\begin{proof}
By Lemma \ref{Lgenalg}, the ideal $L$ is generated as a Lie algebra by the elements $x-x^{\psi}$, for $x \in \mf{g}$. Thus it is enough to show that 
\[ [[y,z^{\psi}], x-x^{\psi}] = 0\] 
for all $x,y,z \in \mf{g}$. We will make repeated use of the fact that $[x,y^{\psi}] = [x^{\psi},y]$ for all $x,y \in \mf{g}$.

Consider the element $[[x,y],z^{\psi}] \in \chi(\mf{g})$, for some $x,y,z \in \mf{g}$. By the Jacobi identity we have
\begin{equation} \label{eqA}
 [[x,y],z^{\psi}] = [[x,z^{\psi}],y] + [x, [y,z^{\psi}]].
\end{equation}
On the other hand, as $[[x,y],z^{\psi}] = [[x,y]^{\psi},z] = [[x^{\psi},y^{\psi}], z]$, we have
\begin{equation} \label{eqB}
 [[x,y],z^{\psi}] = [[x^{\psi},z],y^{\psi}] + [x^{\psi}, [y^{\psi},z]].
\end{equation}
By subtracting \eqref{eqB} from \eqref{eqA} we obtain 
\[ 0 = [[x,z^{\psi}],y-y^{\psi}] + [x-x^{\psi}, [y,z^{\psi}]],\]
and then
\begin{equation}  \label{id1}
 [[x,z^{\psi}],y-y^{\psi}] = [[y,z^{\psi}],x-x^{\psi}]
\end{equation}
for all $x,y,z \in \mf{g}$.

Now
\[\begin{array}{lcl}
[[x,z^{\psi}],y-y^{\psi}] & = & [[x^{\psi},z],y-y^{\psi}] \\
                          & = & -[[z,x^{\psi}],y-y^{\psi}]\\
                          & = & -[[y,x^{\psi}],z-z^{\psi}].
\end{array} 
\]
by \eqref{id1}. If we apply once again this reasoning we get
\[\begin{array}{lcl}
[[x,z^{\psi}],y-y^{\psi}] & = & -[[y,x^{\psi}],z-z^{\psi}] \\
                          & = & [[x,y^{\psi}],z-z^{\psi}] \\
                          & = & [[z,y^{\psi}],x-x^{\psi}].
\end{array} 
\]
The equality above, together with \eqref{id1}, gives us that 
\[ [[y,z^{\psi}],x-x^{\psi}] = - [[y,z^{\psi}],x-x^{\psi}] \]
for all $x,y,z \in \mf{g}$, which completes the proof, since $char(K) \neq 2$.
\end{proof}

The lemma above tells us that $W$ is actually a central subalgebra of $L+D$. In particular, it is an abelian ideal of $\chi(\mf{g})$ as 
we announced.

Now we analyze the quotient $\chi(\mf{g})/W$ or, equivalently, the image $Im(\rho)$.

Denote by $p_1$, $p_2$ and $p_3$ the projections of $\mf{g} \oplus \mf{g} \oplus \mf{g}$ 
onto its first, second and third coordinate, respectively. For any $x \in \mf{g}$ we have 
\[x = p_1( \rho(x)) = p_2 (\rho(x)) = p_3( \rho(x^{\psi})),\]
thus the image $Im(\rho)$ is a \textit{subdirect sum} in $\mf{g} \oplus \mf{g} \oplus \mf{g}$. We will see in the next proposition that actually
$Im(\rho)$ projects surjectively simultaneously on any two copies of $\mf{g}$, which allows us to get some information about finiteness properties.
This will prove Theorem \ref{T1}.

\begin{prop}  \label{Imrhofp}
If $\mf{g}$ is finitely presentable (resp. of type $FP_2$), then $Im(\rho)$ is finitely presented (resp. of type $FP_2$) as well. 
\end{prop}

\begin{proof}
Suppose first that $\mf{g}$ is finitely presentable. For convenience, in this proof we denote by $\mf{g}_1 \oplus \mf{g}_2 \oplus \mf{g}_3$ the range of $\rho$. 
Denote by $p_{(i,j)}: \mf{g}_1 \oplus \mf{g}_2 \oplus \mf{g}_3 \twoheadrightarrow \mf{g}_i \oplus \mf{g}_j$ the projection 
for $(i,j) \in \{ (1,2), (1,3), (2,3)\}$. For any $(x,y) \in \mf{g} \oplus \mf{g}$ we have
\[\begin{array}{lcl}
(x,y) & = & p_{1,2}( \rho(x-x^{\psi}+y^{\psi}))  \\
      & = & p_{1,3}( \rho(x+y^{\psi})) \\
      & = & p_{2,3}( \rho(x-y+y^{\psi})).
\end{array} 
\]
Thus $p_{(i,j)}(Im(\rho)) = \mf{g}_i \oplus \mf{g}_j$ for all $i,j$. It follows by Corollary D1 in \cite{KochMart} that $Im(\rho)$ will
be finitely presented as soon as $Im(\rho) \cap \mf{g}_i \neq 0$ for $i=1,2,3$ (where $\mf{g}_i$ is seen as a subalgebra of 
$\mf{g}_1 \oplus \mf{g}_2 \oplus \mf{g}_3$).

For any $x,y \in \mf{g}$ we have
\[\rho([x,y^{\psi}]) = [(x,x,0),(0,y,y)] = (0,[x,y],0) \in Im(\rho) \cap \mf{g}_2.\]
So if $\mf{g}' \neq 0$ (so that $[x,y] \neq 0$ for some $x,y$), then $Im(\rho) \cap \mf{g}_2 \neq 0$. In this case we also have
\[ \rho([x,y-y^{\psi}]) = ([x,y],0,0) \in Im(\rho) \cap  \mf{g}_1\]
and 
\[ \rho([x^{\psi},y^{\psi}-y] = (0,0,[x,y]) \in Im(\rho) \cap  \mf{g}_3,\]
so that $Im(\rho) \cap \mf{g}_i \neq 0$ also for $i=1$ and $i=3$, as we wanted. If $\mf{g}'$ is actually trivial, then it is easily seen
that $Im(\rho)$ is abelian. But it is also finitely generated, because $\chi(\mf{g})$
is, so $Im(\rho)$ is again finitely presented in this case.

The proof for $\mf{g}$ of type $FP_2$ is the same, except that we use Corollary F1 of \cite{KochMart} instead of Corollary D1.
\end{proof}
Notice that $\mf{g}$ is a retract of $Im(\rho)$, so the converse to the proposition above also holds.
\end{section}

\begin{section}{A condition that forces \texorpdfstring{$W$}{W} to be finite dimensional}

We fix $\mf{g}$ and we study $W=W(\mf{g})$ by means of the extension $W \rightarrowtail L \twoheadrightarrow \rho(L)$. Since $W$ is central in $L$, the associated 5-term
 exact sequence can be written as
\[  H_2(L;K) \to H_2(\rho(L);K) \to W \to H_1(L;K) \to H_1(\rho(L);K) \to 0. \]
It follows that $W$ is finite dimensional if both $H_2(\rho(L);K)$ and $H_1(L;K) \simeq L/L'$ are. 

\subsection{Bounding the dimension of \texorpdfstring{$H_2(\rho(L);K)$}{H2(rho(L);K)}}
Notice that $\rho(L)$ is the subalgebra of $\mf{g} \oplus \mf{g} \oplus \mf{g}$ generated by the elements of the form
\[ \rho(x-x^{\psi}) = (x,x,0) - (0,x,x) = (x,0,-x)\]
for $x \in \mf{g}$. We can identify it with the following subalgebra of $\mf{g} \oplus \mf{g}$:
\[\rho(L) \simeq S:= \langle \{(x,-x) \hbox{ }| \hbox{ } x \in \mf{g}\} \rangle \subseteq \mf{g} \oplus \mf{g}.\]

It is clear that $S$ is a subdirect sum in $\mf{g} \oplus \mf{g}$, but it is not in general of type $FP_2$ (which would imply that $H_2(S;K)$ 
is finite dimensional). For instance, this is not the case if $\mf{g}$ is free non-abelian by Theorem A of \cite{KochMart}. We need to impose 
some restrictions.

\begin{lemma}  \label{homolfg}
If $\mf{g}$ is of type $FP_2$ and $\mf{g}'/\mf{g}''$ is finite dimensional, then $H_2(S;K)$ is finite dimensional as well.
\end{lemma}

\begin{proof}
 For any $x,y \in \mf{g}$, we have $[(x,-x),(y,-y)] = ([x,y],[x,y]) \in S$. Thus 
 \[([x,y],0) = \frac{1}{2}(([x,y],[x,y])+([x,y],-[x,y])) \in S.\] 
 Similarly, $(0,[x,y]) \in S$, so $\mf{g}' \oplus \mf{g}' \subseteq S$.
 
 Notice that $S/(\mf{g}' \oplus \mf{g}') \simeq \mf{g}/\mf{g}'$, and an isomorphism can be given by projection on the first coordinate.
 Consider the Lyndon-Hochschild-Serre spectral sequence associated to this quotient:
 \[ E_{p,q}^2 = H_p(\mf{g}/\mf{g}'; H_q(\mf{g}' \oplus \mf{g}'; K)) \Rightarrow H_{p+q}(S;K).\]
If we want to show that $H_2(S;K)$ is finite dimensional, it is enough to show that $E_{p,q}^2$ is finite dimensional for all $p,q \geq 0$ with $p+q =2$.

First consider $(p,q) = (2,0)$. Clearly $H_0(\mf{g}' \oplus \mf{g}'; K) \simeq K$, so $E_{2,0}^2 \simeq H_2(\mf{g}/\mf{g}'; K)$. But $\mf{g}/\mf{g}'$
is finite dimensional, hence of type $FP_{2}$, so $H_2(\mf{g}/\mf{g}'; K)$ is finite dimensional.

Let $(p,q)=(1,1)$. First of all, $H_1(\mf{g}' \oplus \mf{g}'; K) \simeq \mf{g}'/\mf{g}'' \oplus \mf{g}'/\mf{g}''$. The action of 
$\mf{g}/\mf{g}'$ on  $H_1(\mf{g}' \oplus \mf{g}'; K)$ is then converted in an action by $\mf{g}/\mf{g}'$ on 
$\mf{g}'/\mf{g}'' \oplus \mf{g}'/\mf{g}''$, which is induced by the adjoint action on the first coordinate and the same on the second
coordinate, but with opposite sign.

Now, $\mf{g}'/\mf{g}''$ is clearly finitely generated as a $\mf{g}/\mf{g}'$-module, and 
since $\mf{g}/\mf{g}'$ is of finite dimension,
$\mathcal{U}(\mf{g}/\mf{g}')$ is noetherian. This implies that $\mf{g}'/\mf{g}'' \oplus \mf{g}'/\mf{g}''$ is
actually of type $FP_{\infty}$ as a $\mathcal{U}(\mf{g}/\mf{g}')$-module. 
Thus $H_1(\mf{g}/\mf{g}'; H_1(\mf{g}' \oplus \mf{g}'; K)) \simeq H_1(\mf{g}/\mf{g}'; \mf{g}'/\mf{g}'' \oplus \mf{g}'/\mf{g}'')$
is finite dimensional.
 
Finally, let $(p,q)=(0,2)$. Now we want to show that $H_0(\mf{g}/\mf{g}'; H_2(\mf{g}' \oplus \mf{g}'; K))$ is finite dimensional. By the Künneth formula 
\[ H_2(\mf{g}' \oplus \mf{g}'; K) \simeq \oplus_{0 \leq i \leq 2} (H_i(\mf{g}';K) \otimes_K H_{2-i}(\mf{g}';K)).\]
Clearly it is enough to show that each of the components of the direct sum above is finitely generated as a $\mf{g}/\mf{g}'$-module.

One of the components is
\[ H_1(\mf{g}';K) \otimes_K H_1(\mf{g}';K) \simeq \mf{g}'/\mf{g}'' \otimes_K \mf{g}'/\mf{g}'',\]
which is clearly finitely generated, since $\mf{g}'/\mf{g}''$ is of finite dimension.

The other two components are isomorphic and we have
\[H_2(\mf{g}';K) \otimes_K H_0(\mf{g}';K) \simeq H_2(\mf{g}';K) \otimes_K K \simeq H_2(\mf{g}';K).\]
Consider a projective resolution
\[ \mathcal{P}: \ldots \to P_n \to P_{n-1} \to  \ldots \to P_2 \to P_1 \to P_0 \to K \to 0\]
of $K$ as a $\mathcal{U}(\mf{g})$-module, with $P_j$ finitely generated for $j \leq 2$. By applying the functor $K \otimes_{\mathcal{U}(\mf{g}')} -$,
we get a complex of $\mathcal{U}(\mf{g}/\mf{g}')$-modules, and the modules are still finitely generated up to dimension $2$. As $\mathcal{U}(\mf{g}/\mf{g}')$
is noetherian, the homologies of this complex are also finitely generated up to dimension $2$. But
\[ H_i(K \otimes_{\mathcal{U}(\mf{g}')} \mathcal{P}) \simeq H_i(\mf{g}'; K).\]
Thus $H_2(\mf{g}';K)$ is finitely generated as a $\mathcal{U}(\mf{g}/\mf{g}')$-module, as we wanted.
\end{proof}

\subsection{Bounding the dimension of \texorpdfstring{$L/L'$}{L/L')}}
 Lima and Oliveira proved in \cite{LimaOliveira} that, in the 
 group-theoretic case, the quotient $L/L'$ is finitely generated as soon as the original group 
 $G$ is finitely generated. The idea was to realize $L/L'$ as a certain quotient of the augmentation ideal $Aug(\mathbb{Z}G)$, where a finite 
 generating set could be detected more easily. We adapt their argument.
 
 Recall that we denote by $\mathcal{U}(\mf{g})$ the universal enveloping algebra of $\mf{g}$. It can be seen as an abelian
 Lie algebra, that is, a vector space with trivial bracket. There is an action of $\mf{g}$ on it by left multiplication in the associative sense, that is, for any $x \in \mf{g}$ and 
 $x_1 \cdots x_n \in \mathcal{U}(\mf{g})$ a monomial, the action is given by
 \[ x \cdot (x_1 \cdots x_n) = x x_1 \cdots x_n.\] 
 The augmentation ideal $Aug(\mathcal{U}(\mf{g}))$ is clearly invariant by this action,
 so we can consider the semi-direct product $\Gamma = Aug(\mathcal{U}(\mf{g})) \rtimes \mf{g}$. Our convention is that, for $u_1, u_2 \in Aug(\mathcal{U}(\mf{g}))$
 and $x_1, x_2 \in \mf{g}$, the bracket is given by
 \[ [(u_1,x_1), (u_2,x_2)] = (x_1 u_2 - x_2 u_1, [x_1,x_2]).\]
 
 Consider the Lie algebra presented by
 \[ \Delta = \langle \mf{g}, \mf{g}^{\psi} \hbox{ } | \hbox{ } [\mc{L}, \mc{L}] = 0 \rangle,\]
 where $\mc{L} = \langle \langle x-x^{\psi}; x \in \mf{g} \rangle \rangle$, an ideal of the free Lie sum of $\mf{g}$ and $\mf{g}^{\psi}$. 
 Define $\theta: \Delta \to \Gamma$ by 
 \[ \theta(x) = (0,x) \hbox{  and  } \theta(x^{\psi}) = (-x,x).\]
 It is not hard to check this defines $\theta$ as a homomorphism of Lie algebras.
 Notice that for any $x_1, \ldots, x_n \in \mf{g}$ we have
 \[ \theta(x_1-x_1^{\psi}) = (x_1, 0)\]
 and
 \[ \theta([x_1, \ldots, x_{n-1}, x_n-x_n^{\psi}]) = (x_1 \cdots x_n, 0),\]
 where $[x_1, \ldots, x_{n-1}, x_n-x_n^{\psi}]$ is a right-normed bracket. So $\theta(\mc{L}) = Aug(\mc{U}(\mf{g}))$. Also, $\theta$ is surjective, as its image contains a basis for 
 $Aug(\mathcal{U}(\mf{g}))$, as well as all elements of the form $(0,x)$, for $x \in \mf{g}$. 
 
 \begin{lemma}
  $\theta$ is injective.
 \end{lemma}
 \begin{proof}
  First notice that $\Delta = \mc{L} \rtimes \mf{g}$. Since $\theta$ restricts to the identity on $\mf{g}$, we only need to show that 
  $\theta |_{\mc{L}}: \mc{L} \to Aug(\mc{U}(\mf{g}))$ is injective.

 Let $\{x_j\}_{j \in J}$ be an ordered basis of $\mf{g}$. Observe that $\mc{L}$ is linearly spanned by the brackets 
 $[y_1, \ldots, y_{n-1}, x_j-x_j^{\psi}]$, 
 with $y_i \in \mf{g} \cup \mf{g}^{\psi}$ and $j \in J$. Moreover, since $[\mc{L},\mc{L}]=0$, we have $[y, \ell] = [y^{\psi}, \ell]$ 
 for all $\ell \in \mc{L}$ and $y \in \mf{g}$, so we can 
 actually assume that each $y_i$ lies in $\mf{g}$. Finally, we can pass to the elements of the basis: $\mc{L}$ is spanned by 
 $[x_{j_1}, \ldots, x_{j_{n-1}}, x_{j_n}-x_{j_n}^{\psi}]$ with $j_1, \ldots, j_n \in J$.

 For any $j_1, \ldots, j_n \in J$, let 
 \[\ell(x_{j_1}, \ldots, x_{j_n}) := [x_{k_1}, \ldots, x_{k_{n-1}}, x_{k_n}-x_{k_n}^{\psi}],\]
 where $\{k_1, \ldots, k_n\} = \{j_1, \ldots, j_n\}$ and $k_1 \leq \ldots \leq k_n$.
 
 Let $A_m$ be the subspace spanned by all possible $\ell(x_{j_1}, \ldots, x_{j_n})$ with $n \leq m$. Notice that $\theta$ takes the set of the
 $\ell(x_{j_1}, \ldots, x_{j_m})$ to a basis of $Aug(\mc{U}(\mf{g}))$, as described by the Poincaré-Birkhoff-Witt theorem.
 Thus to show that $\theta$ is injective, it is enough to prove that $\mc{L} = \sum_m A_m$. To see that this is the case we use the following claim, which 
 we prove by induction.
 
 \textbf{Claim:} For any $k, j_1, \ldots, j_n \in J$, there is some $a \in A_n$ such that 
 \[[x_k, \ell(x_{j_1}, \ldots, x_{j_n})] =a + \ell(x_k, x_{j_1}, \ldots, x_{j_n}).\]
 
 Suppose $n=1$. If $k \leq j_1$, then $[x_k, \ell(x_{j_1})] = \ell(x_k, x_{j_1})$ by definition. Otherwise, notice that 
 \[ [x_k, \ell(x_{j_1})] = [x_k^{\psi}, \ell(x_{j_1})] = [x_k^{\psi}, x_{j_1}-x_{j_1}^{\psi}] = [x_{j_1},x_k-x_k^{\psi}] + ([x_k,x_{j_1}] - [x_k,x_{j_1}]^{\psi})  \]	
 The first equality comes from the fact that $[\mc{L}, \mc{L}]=0$. But $[x_{j_1},x_k-x_k^{\psi}] = \ell(x_k, x_{j_1})$ and $[x_k,x_{j_1}] - [x_k,x_{j_1}]^{\psi}$
 lies in $A_1$, as we can see by rewriting $[x_k,x_{j_1}]$ in terms of the basis elements of $\mf{g}$. So we are done in this case.
 
 Now suppose that $n> 1$ and that the claim holds for smaller values of $n$. Let $k, j_1, \ldots, j_n \in J$ with $j_1 \leq \ldots \leq j_n$. If
 $k \leq j_1$, then $[x_k, \ell(x_{j_1}, \ldots, x_{j_n})] = \ell(x_k, x_{j_1}, \dots, x_{j_n})$ by definition. Otherwise, by the Jacobi identity we have:
  \[ [x_k, \ell(x_{j_1}, \ldots, x_{j_n})] = [[x_k,x_{j_1}], [x_{j_2}, \ldots, x_{j_n}-x_{j_n}^{\psi}]] + [x_{j_1}, [x_k, [x_{j_2}, \ldots, x_{j_n}-x_{j_n}^{\psi}]]] \]
 Let $[x_k,x_{j_1}] = \sum_{\alpha} \lambda_{\alpha} x_{\alpha}$, with $\lambda_{\alpha} \in K$ and $\alpha \in J$. Then 
 \[[[x_k,x_{j_1}], [x_{j_2}, \ldots, x_{j_n}-x_{j_n}^{\psi}]] = \sum_{\alpha} \lambda_{\alpha} [x_{\alpha}, \ell(x_{j_2}, \ldots, x_{j_n})]=:a.\]
 By induction hypothesis $a \in A_n$. Also by induction hypothesis we can write 
 \[ [x_k, [x_{j_2}, \ldots, x_{j_n}-x_{j_n}^{\psi}]] = b + \ell(x_k, x_{j_2}, \ldots, x_{j_n}), \]
 for some $b \in A_{n-1}$. Notice that $[x_{j_1},b] \in A_n$ (again by induction hypothesis) and 
 $ [x_{j_1}, \ell(x_k, x_{j_2}, \ldots, x_{j_n})] = \ell(x_{j_1}, x_k, x_{j_2}, \ldots, x_{j_n})=\ell(x_k, x_{j_1}, \ldots, x_{j_n})$,
 since $j_1 <k$ and $j_1 \leq j_t$ for $2 \leq t \leq n$. Thus:
  \[ [x_k, \ell(x_{j_1}, \ldots, x_{j_n})] = a + [x_{j_1}, b] + \ell(x_k, x_{j_1}, \ldots, x_{j_n}),\]
  with  $a + [x_{j_1}, b] \in A_n$, as we wanted. The claim is proved.
  
  Finally, $\mc{L} = \sum_n A_n$. Indeed, it is clear that $\mc{L} \supseteq \sum_n A_n$. For the reverse inclusion, recall 
  that $\mc{L}$ is linearly spanned by the brackets $w = [x_{j_1}, \ldots, x_{j_{n-1}}, x_{j_n}-x_{j_n}^{\psi}]$ 
  with $j_1, \ldots, j_n \in J$. It follows by the claim that $w$ and $\ell(x_{j_1}, \ldots, x_{j_{n-1}}, x_{j_n})$ differ only by an 
  element of $A_n$, and $\ell(x_{j_1}, \ldots, x_{j_{n-1}}, x_{j_n}) \in A_n$ by definition, so we are done. 
 \end{proof}

 So $\theta: \Delta \to \Gamma$ is an isomorphism. Now we introduce the relations of $\chi(\mf{g})$. Notice that for $x \in \mf{g}$ we have
 \[ \theta([x,x^{\psi}]) = [(0,x),(-x,x)] = (-x^2,0).\]
 Denote by $I$ the smallest ideal of $Aug(\mc{U}(\mf{g})) \rtimes \mf{g}$ containing $(x^2,0)$ for all $x \in \mf{g}$.  Notice that 
 $I \subseteq Aug(\mathcal{U}(\mf{g}))$ and that $I$ is invariant by the action of $\mf{g}$, so the semi-direct 
 product $\frac{Aug(\mathcal{U}(\mf{g}))}{I} \rtimes \mf{g}$ still makes sense. It follows that $\theta$ induces an isomorphism
 \[ \bar{\theta}: \frac{\chi(\mf{g})}{L'} \to \frac{Aug(\mathcal{U}(\mf{g}))}{I} \rtimes \mf{g},\]
 and restriction induces an isomorphism of $L/L'$ and $Aug(\mathcal{U}(\mf{g}))/I$.
 
 \begin{rem}
  Notice that $I$ is the two-sided ideal of $\mathcal{U}(\mf{g})$, \textit{as an associative algebra}, generated by the elements $x^2$, for 
 $x \in \mf{g}$. This follows by the fact that $I$ is closed by left multiplication by elements of $\mf{g}$, by construction, but also by right 
 multiplication, which can be deduced from the identity \eqref{eqcomm1} in the proof below.
 \end{rem}

 \begin{prop} \label{Labfindim}
 If $\mf{g}$ is a finitely generated Lie algebra, then $L/L'$ is finite dimensional.
 \end{prop}
 
 \begin{proof}
 We must show that $Aug(\mathcal{U}(\mf{g}))/I$ is finite dimensional. First notice that for any $x,y \in \mf{g}$, the following holds modulo $I$:
 \[0 = (x+y)^2 = x^2 + xy + yx + y^2 = xy + yx,\]
 so
 \begin{equation} \label{eqcomm1}
 xy= -yx  \hbox{ (mod } I)  
 \end{equation}
 for all $x, y \in \mf{g}$. Also
 \begin{equation} \label{eqcomm2}
 [x,y] = xy-yx = 2xy \hbox{ (mod } I) 
 \end{equation}
 for all $x,y \in \mf{g}$.
 
 Notice that $Aug(\mathcal{U}(\mf{g}))/I$ is generated as a vector space by the image of $\mf{g}$ under the projection 
 $Aug(\mathcal{U}(\mf{g})) \twoheadrightarrow Aug(\mathcal{U}(\mf{g}))/I$. Indeed for any monomial $x_1 \cdots x_n$, with $n \geq 2$, we have
 by \eqref{eqcomm2}
 \[ x_1 \cdots x_n  =  x_1 \cdots x_{n-2} (\frac{1}{2} [x_{n-1},x_n]) = x_1 \cdots x_{n-2}x_{n-1}' \hbox{ (mod } I) .\]
 where $x_{n-1}' = \frac{1}{2}[x_{n-1},x_n]$. By induction we see that $x_1 \cdots x_n$ is congruent modulo $I$ to 
 $\frac{1}{2^{n-1}}[x_1, \ldots, x_n]$ (again a right-normed bracket), which is the image of an element of $\mf{g}$.
 
 Now suppose that $\mf{g} = \langle x_1, \ldots, x_n \rangle$. Then any $x \in \mf{g}$ can be written as a linear combination of arbitrary brackets
 of any length involving the $x_i$'s. If $c$ is any of these brackets, then writing it associatively in $Aug(\mathcal{U}(\mf{g}))$ we obtain a sum of elements of the form
 \[ \lambda x_{i_1} \cdots x_{i_m} \]
 with $\lambda \in K$, $i_j \in \{1, \ldots, n\}$ and $m \geq 1$. By commuting the variables (using identity \eqref{eqcomm1}), we get that, modulo
 $I$, $c$ is a sum 
 of elements of the form
 \[ \mu x_1^{t_1} \cdots x_n^{t_n} \]
 with $\mu \in K$ and $t_j \geq 0$ for all $j$. By the definition of $I$, this term is $0$ in $Aug(\mathcal{U}(\mf{g}))/I$ if $t_j \geq 2$
 for some $j$. It follows that  $Aug(\mathcal{U}(\mf{g}))/I$ is actually generated by the arbitrary brackets involving $x_1, \ldots, x_n$ such that 
 no $x_j$ is repeated. Thus the brackets of length not greater than $n$ are sufficient, whence $Aug(\mathcal{U}(\mf{g}))/I$ is finite dimensional.
 \end{proof}

This completes the proof of Theorem \ref{T2}.

\begin{rem} \label{mnu}
 Notice that for any $u,v,w \in \mf{g}$, the following holds in $\chi(\mf{g})$:
 \[[u,[v,w]]-[u^{\psi},[v^{\psi},w^{\psi}]] =  [u-u^{\psi}, [v-v^{\psi}, w -w^{\psi}]].\]
 This can be proved using the identities of the form $[u,v^{\psi}]=[u^{\psi},v]$ and $[D,L]=0$. Now, if $\mf{g} = \langle x_1, \ldots x_n\rangle$, 
 then $L/L'$ is generated by the elements $x_i-x_i^{\psi}$ and $[x_i,x_j]-[x_i,x_j]^{\psi}$ with $1 \leq i < j \leq n$, since
 $m-m^{\psi} \in L'$ for any monomial $m$ of degree at least $3$. This gives a simpler proof of Proposition \ref{Labfindim}, but we decided to keep 
 the one above because it offers extra insight on the structure of $\chi(\mf{g})$.
\end{rem}

\end{section}

\begin{section}{Finiteness properties of \texorpdfstring{$\chi(\mf{g})$}{chi(g)}}

  In the last section we showed that $L/L'$ is of finite dimension as soon as $\mf{g}$ is finitely generated, but it turns out that actually $L$
  is already finitely generated. Consider the exact sequence given by
  \[ 0 \to W  \to L \to \rho(L) \to 0.\]
  Notice that $\rho(L)$ is finitely generated. Indeed, if $x_1, \ldots, x_n$ generate $\mf{g}$, then the elements $(x_i,0,-x_i)$ and $([x_j,x_k],0,0)$ for $1 \leq i \leq n$
  and $1 \leq j < k \leq n$ generate $\rho(L)$. 
 
 \begin{prop}
  If $\mf{g}$ is finitely generated, then $L=L(\mf{g})$ is finitely generated.
 \end{prop}
 
 \begin{proof}
 Let $T \subseteq L$ be a finite set whose image generates $\rho(L)$. The fact that $W$ is central implies that the subalgebra 
 $\langle T \rangle$ is actually an ideal. Indeed, any $\ell \in L$ is written as $\ell = \tau + w$ for some $\tau \in \langle T \rangle$ and $w \in W$,
 so 
 \[ [t,\ell] = [t,\tau] + [t,w] = [t,\tau] \in \langle T \rangle\]
 for any $t \in T$.
 
 The quotient $L/\langle T \rangle$ is the image of $W$ by the canonical projection, therefore is abelian. But then it is also
 finitely dimensional, being a quotient of $L/L'$. The exactness of $\langle T \rangle \rightarrowtail L \twoheadrightarrow L/\langle T \rangle$ implies that
 $L$ is itself finitely generated.
 \end{proof}

\subsection{Finite presentability}
The first step towards a proof of the first part of Theorem \ref{T3} is to establish it for free Lie algebras.

\begin{prop}  \label{fpfree}
If $\mf{f}$ is a free Lie algebra of finite rank, then $\chi(\mf{f})$ is finitely presented.
\end{prop}

\begin{proof}
Recall that $\chi(\mf{f}) \simeq L \rtimes \mf{f}$.
Let $\{x_1, \ldots, x_m\}$ be a free basis for $\mf{f}$ and let $L = \langle \ell_1, \ldots, \ell_n \rangle$. Notice that 
$\chi(\mf{f})/D \simeq \mf{f} \oplus \mf{f}$,
which is a finitely presented Lie algebra. Since $\chi(\mf{f})$ is finitely generated, $D$ is finitely generated as an ideal. Let
$D = \langle \langle d_1, \ldots, d_s \rangle \rangle$. Each $d_i$ can be written as a sum of brackets involving the generators 
$\ell_1, \ldots, \ell_n, x_1, \ldots, x_m$;
we denote by $\delta_i$ one of such sums.

Similarly, $\chi(\mf{f})/W \simeq Im(\rho)$ is finitely presented by Proposition \ref{Imrhofp},
thus we can write
\[Im(\rho) = \langle \ell_1, \ldots, \ell_n, x_1, \ldots, x_m \hbox{ } | \hbox{ } \tau_1, \ldots, \tau_k \rangle\]
for some $\tau_i$'s. If we denote by $F$ the free lie algebra on $\{\ell_1, \ldots, \ell_n, x_1, \ldots, x_m\}$, then the obvious homomorphism
$F \twoheadrightarrow Im(\rho)$ is injective on the subalgebra generated by $\{x_1, \ldots, x_m\}$, as $\mf{f}$ is free on this set. 
It follows that each $\tau_i$ is an element of the ideal of $F$ generated by $\{\ell_1, \ldots, \ell_n\}$. 

Finally, since $L$ is an ideal, we can choose words $\mu_{i,j}$ in $\ell_1, \ldots, \ell_n$ representing $[x_i, \ell_j]$, for each $i,j$.

Let $\Gamma$ be the Lie algebra generated by the symbols $\ell_i, x_i, d_i, w_j$, where $i$ runs through the appropriate indices and $1 \leq j \leq k$, subject to the following defining relations:
\begin{enumerate}
 \item $d_i=\delta_i$ for $1 \leq i \leq s$;
 \item $w_i= \tau_i$ for $1 \leq i \leq k$;
 \item $[x_i,\ell_j] = \mu_{i,j}$ for $1 \leq i \leq m$ and $1 \leq j \leq n$;
 \item $[d_i, \ell_j]=0$ for $1 \leq i \leq s$ and $1 \leq j \leq n$;
 \item $[w_i,\ell_j]=0$ for $1 \leq i \leq k$ and $1 \leq j \leq n$.
\end{enumerate}

Denote by $L_0$ the subalgebra of $\Gamma$ generated by $\{\ell_i| 1 \leq i \leq n\}$, by $D_0$ the ideal generated by 
$\{d_i | 1 \leq i \leq s\}$ and by
$W_0$ the ideal generated by $\{w_i | 1 \leq i \leq k\}$. Notice that $L_0$ is actually an ideal of $\Gamma$, by the relations of types 3, 4 and 5, together with the definition of the words $\mu_{i,j}$. The relations of type 4 imply that $[D_0,L_0] = 0$. From the relations of type 5 we conclude that $W_0$ commutes with $L_0$, while the relations of the types 2, 3 and 4 imply that $W_0$ is commutes with $D_0$, since each $w_i$ represents an element of the ideal generated by $\ell_1, \ldots, \ell_n$, that is, the subalgebra $L_0 \subseteq \Gamma$, which commutes with $D_0$. So $W_0$ is central in $L_0+D_0$, and in particular it is an abelian ideal of $\Gamma$.

It is clear that there is a well-defined surjective homomorphism $\phi: \Gamma \to \chi(\mf{f})$ that takes the generators of $\Gamma$ to the corresponding
elements in $\chi(\mf{f})$. The choice of $\tau_1, \ldots, \tau_k$ implies that $\phi$ induces an isomorphism $\Gamma/W_0 \simeq \chi(\mf{f})/W$. Thus $ker(\phi) \subseteq W_0$.
Also $\phi(L_0) = L$ and $\phi(D_0)=D$, and since $W_0 \subseteq L_0 + D_0$, we have
\[\frac{\Gamma}{L_0 + D_0} \simeq \frac{\chi(\mf{f})}{L+D} \simeq \frac{\mf{f}}{\mf{f}'}.\]

Now $W_0$ is a module over the universal enveloping algebra of $\Gamma/(L_0+D_0) \simeq \mf{f}/\mf{f}'$, and it is generated by $w_1, \ldots, w_k$.
The fact that $\mathcal{U}( \mf{f}/\mf{f}')$ is noetherian implies that $ker(\phi)$, being a submodule of $W_0$, is finitely generated too. But
then $\chi(\mf{f}) \simeq \Gamma/ker(\phi)$ is finitely presented.
\end{proof}

\begin{crlr}  \label{chifp}
If $\mf{g}$ is finitely presented, then so is $\chi(\mf{g})$. 
\end{crlr}

\begin{proof}
 Let $\mf{g} = \mf{f}/\langle \langle R \rangle \rangle$, where $\mf{f}$ is a free Lie algebra of finite rank and $R$ is a finite set.
 By Proposition \ref{fpfree} we know that $\chi(\mf{f})$ is finitely presented. But then
 \[ \chi(\mf{g}) \simeq \chi(\mf{f})/ \langle \langle r, r^{\psi}; \hbox{ for }r \in R \rangle \rangle\]
 is also finitely presented. 
\end{proof}

The converse to the corollary above is clearly also true since $\mf{g}$ is a retract of $\chi(\mf{g})$. 

\subsection{Property \texorpdfstring{$FP_2$}{FP2}}
The version of this result for the property $FP_2$ can be obtained from the lemmas in Section \ref{sec2}. 

\begin{prop}
If $\mf{g}$ is of type $FP_2$, then so is $\chi(\mf{g})$.
\end{prop}
\begin{proof}
By Lemma \ref{fpFP2}, there is a finitely presented Lie algebra $\mf{h}$ and an ideal $\mf{r} \subseteq \mf{h}$ such that 
$\mf{g} \simeq \mf{h}/\mf{r}$ and $\mf{r} = [\mf{r},\mf{r}]$. Now, $\chi(\mf{h})$ is finitely presented by
Corollary \ref{chifp}, and clearly $\chi(\mf{g}) \simeq \chi(\mf{h})/I$, where $I$ is the ideal generated by $x$ and $x^{\psi}$, for all $x \in \mf{r}$. This ideal
is perfect, since $\mf{r} = [\mf{r},\mf{r}] \subseteq [I,I]$, and the same holds for $\mf{r}^{\psi}$. Thus $\chi(\mf{g})$ is of type $FP_2$ by Lemma \ref{fpFP2}. 
\end{proof}

Again, the converse is true by the Lemma \ref{retractsFPm}.

\subsection{\texorpdfstring{$\chi(\cdot)$}{chi()} does not preserve \texorpdfstring{$FP_3$}{FP3}}
Fix $\mf{f}$ a non-abelian free Lie algebra of finite rank. Recall that subalgebras of a free Lie algebra are again free (see \cite{Shirshov}).

\begin{lemma}  \label{H2Sinfdim}
Let $S = \langle \{ (x,-x) \hbox{ } | \hbox{ } x \in \mf{f}\}\rangle \subseteq \mf{f} \oplus \mf{f}$.
Then $H_2(S;K)$ is infinite dimensional.
\end{lemma}

\begin{proof}
Let $\pi: \mf{f} \oplus \mf{f} \to \mf{f}$ be the projection onto the second coordinate and let $N=ker(\pi) \cap S$. Notice that
$N \subseteq \mf{f} \oplus 0 \simeq \mf{f}$, so it is a free Lie 
algebra. In fact, it is not hard to see that $N$ is the inclusion of $\mf{f}'$ on the first coordinate of $\mf{f} \oplus \mf{f}$, i.e., 
$N = \mf{f}' \oplus 0$ (see the proof of Lemma \ref{homolfg}).
The sequence $N \rightarrowtail S \twoheadrightarrow \mf{f}$ is exact, so there is a spectral sequence
\[ E_{p,q}^2 = H_p(\mf{f}, H_q(N;K)) \Rightarrow H_{p+q}(S;K).\]
The fact that $\mf{f}$ is free implies that $E_{p,q}^2=0$ for all $p \geq 2$, so $E^2 = E^{\infty}$. 

Suppose on the contrary that $H_2(S;K)$ is finite dimensional. Then the subquotient $E_{1,1}^2$ is also finite dimensional, and 
\[E_{1,1}^2 = H_1(\mf{f}; H_1(N;K)) \simeq H_1(\mf{f}; N/N').\]
Now fix any $c \in \mf{f}' \smallsetminus \{0\}$. Notice that $c$ acts trivially on $N/N'$, so $dim_K (c \cdot N/N') = 0$, where $c$ is identified with
its image in $\mathcal{U}(\mf{f})$. It follows by Lemma \ref{lemma3.1KM} that $N/N'$ is itself finite dimensional, which is turn equivalent to $N$ being
finitely generated. This is a contradiction, since $N \simeq \mf{f}'$, and $\mf{f}$ is a non-abelian free Lie algebra.
\end{proof}

\begin{lemma}  \label{H2Lfdim}
If $\chi(\mf{f})$ is of type $FP_3$, then $H_2(L;K)$ is of finite dimension. 
\end{lemma}

\begin{proof}
 The short exact sequence $L \rightarrowtail \chi(\mf{f}) \twoheadrightarrow \mf{f}$ gives rise to a spectral sequence:
 \[ E_{p,q}^2 = H_p(\mf{f}; H_q(L;K)) \Rightarrow H_{p+q}(\chi(\mf{f});K).\]
 Notice that, as in the proof of Lemma \ref{H2Sinfdim}, the fact that $\mf{f}$ is free implies that  
 $E^2 = E^{\infty}$. Now $H_3(\chi(\mf{f});K)$ is finite dimensional, thus $E_{p,q}^2 = E_{p,q}^{\infty}$ is finite dimensional as well
 whenever $p+q=3$. In particular,  $E_{1,2}^2 = H_1(\mf{f}; H_2(L;K))$ is finite dimensional. 
 
 Now for any $x,y \in \mf{f}$ the element $[x,y]$ acts trivially on $L$, since it is the image of the element $[x,y^{\psi}] \in D \subseteq \chi(\mf{f})$,
 and $[D,L]=0$. It follows that $[x,y]$ acts trivially on $H_2(L;K)$ as well. 
 Finally since $\mf{f}$ is non-abelian, $[x,y]$ can be taken to be non trivial,
 so Lemma \ref{lemma3.1KM} applies again: $H_2(L;K)$ is finite dimensional.
\end{proof}

\begin{prop}
$\chi(\mf{f})$ is not of type $FP_3$.
\end{prop}

\begin{proof}
Suppose on the contrary that $\chi(\mf{f})$ is of type $FP_3$. Consider the spectral sequence associated to $W \rightarrowtail L \twoheadrightarrow S$:
\[ E_{p,q}^2 = H_p(S; H_q(W;K)) \Rightarrow H_{p+q}(L;K).\]
Since $S \subseteq \mf{f} \oplus \mf{f}$, it follows that $E_{p,q}^2=0$ for all $p \geq 3$. 
Consider the term $E_{1,1}^2$. Recall that the bidegree of the differential map of the spectral sequence $\{E^r\}$ is $(-r,r-1)$. Thus the differential maps
that involve $E_{1,1}^2$ are
\[d_{1,1}^2: E_{1,1}^2 \to E_{-1,2}^2\]
and
\[d_{3,0}^2: E_{3,0}^2 \to E_{1,1}^2.\]
Note that $E_{-1,2}^2 = 0 = E_{3,0}^2$, so $E_{1,1}^2 = E_{1,1}^3$. The fact that $E^3$ is non trivial only on the columns 
$p=0,1,2$, together with the knowledge of the bidegree implies that $d^r$ is trivial for $r \geq 3$. Thus $E^3 = E^{\infty}$. It follows that
$E_{1,1}^2 = E_{1,1}^{\infty}$ is a subquotient of $H_2(L;K)$, which is finite dimensional
by Lemma \ref{H2Lfdim}.

On the other hand
\[E_{1,1}^2 = H_1(S; H_1(W;K)) = H_1(S;W).\]
Notice that $S$ acts trivially on $W$, so $H_1(S;W) \simeq S/S' \otimes_K W$. Thus $W$ must be finite dimensional, since $S/S'$ is not trivial ($S$ projects
onto $\mf{f}/\mf{f}'$).

Finally, the 5-term exact sequence associated to $W \rightarrowtail L \twoheadrightarrow S$ can be written as:
\[H_2(L;K) \to H_2(S;K) \to W \to H_1(L;K) \to H_1(S;K) \to 0.\]
But $H_2(L;K)$ and $W$ are both finite dimensional, so $H_2(S;K)$ is finite dimensional as well. This contradicts Lemma \ref{H2Sinfdim}.
\end{proof}
\end{section}

\begin{section}{Stem extensions and the Schur multiplier}
Given any Lie algebra $\mf{g}$, denote 
\[R = R(\mf{g}) := [\mf{g}, L, \mf{g}^{\psi}] \subseteq \chi(\mf{g}).\]
This is the subalgebra of $\chi(\mf{g})$ generated by the triple brackets $[x,[\ell,y^{\psi}]]$, for all $x,y \in \mf{g}$ and $\ell \in L=L(\mf{g})$. It follows from 
the facts that $L$ is an ideal and $[L,D]=0$ that $R$ is actually an ideal of $\chi(\mf{g}$). Notice also that $R \subseteq W=W(\mf{g})$.

Recall that $D$ is the ideal of $\chi(\mf{g})$ generated by the elements $[y,z^{\psi}]$. In general it is not generated by these elements as a Lie subalgebra,
but it will be modulo $R$. Indeed, notice that for $x,y,z \in \mf{g}$ we have
\[  [x,[y-y^{\psi},z^{\psi}]] = [x,[y,z^{\psi}]] - [x,[y^{\psi},z^{\psi}]] = [x,[y,z^{\psi}]] - [x,[y,z]^{\psi}]\]
Since $[x,[y-y^{\psi},z^{\psi}]] \in R$, it follows that $[x,[y,z^{\psi}]]$ is congruent to $[x,[y,z]^{\psi}]$. The same holds for 
$[x^{\psi},[y,z^{\psi}]]$. Thus $D/R$ is actually generated as an algebra by the image of the brackets $[x,y^{\psi}]$, for $x,y \in \mf{g}$.

Now we consider the quotient $W/R$. Since $W \subseteq D$, it follows by the comments above that the elements of $W/R$ are of the form:
\begin{equation}  \label{WR1}
  w + R = \sum_{\alpha} \lambda_{\alpha} [ [x_{\alpha,1},y_{\alpha,1}^{\psi}], \dots, [x_{\alpha,n_{\alpha}},y_{\alpha,n_{\alpha}}^{\psi}]]+R,
\end{equation}
with $\lambda_{\alpha} \in K$ and $x_{\alpha,j}, y_{\alpha,j} \in \mf{g}$.
Also, as $W \subseteq L$, it must be true as well that
\begin{equation}  \label{WR2}
\sum_{\alpha} \lambda_{\alpha} [ [x_{\alpha,1},y_{\alpha,1}], \ldots, [x_{\alpha,n_{\alpha}},y_{\alpha,n_{\alpha}}]]=0,
\end{equation}
and this describes completely the elements of $W/R$.

\subsection{\texorpdfstring{$W/R$}{W/R} is a quotient of \texorpdfstring{$H_2(\mf{g};K)$}{H2(g;K)}}
Following Ellis \cite{Ellis}, we consider the non-abelian exterior product $\mf{g} \wedge \mf{g}$. It is defined as the Lie algebra generated by the symbols
$x \wedge y$, with $x,y \in \mf{g}$, subject to the following defining relations:
\begin{enumerate}
 \item $(x_1 + x_2) \wedge y = x_1 \wedge y + x_2 \wedge y$;
 \item $x \wedge (y_1 + y_2) = x \wedge y_1 + x \wedge y_2$;
 \item $\lambda (x \wedge y) = (\lambda x) \wedge y = x \wedge (\lambda y)$;
 \item $x \wedge x = 0$;
 \item $[x_1,x_2] \wedge y = [x_1,y] \wedge x_2 + x_1 \wedge [x_2,y]$;
 \item $x \wedge [y_1,y_2] = [x,y_1] \wedge y_2 + y_1 \wedge [x,y_2]$;
 \item $[x_1 \wedge y_1, x_2 \wedge y_2] = [x_1, y_1] \wedge [x_2,y_2]$;
\end{enumerate}
for all $x,x_1,x_2,y,y_1,y_2 \in \mf{g}$ and $\lambda \in K$. 
  
Let $\phi: \mf{g} \wedge \mf{g} \to \mf{g}$ be the Lie algebra homomorphism defined by $\phi(x \wedge y) = [x,y]$. The main result in \cite{Ellis} is that
$ker(\phi)$ is isomorphic to the Schur multiplier $H_2(\mathfrak{g}; K)$.

Notice that an element in $ker(\phi)$ is written
as 
\begin{equation} \label{WR3}
 \sum_{\alpha} \lambda_{\alpha} [ x_{\alpha,1} \wedge y_{\alpha,1}, \ldots, x_{\alpha,n_{\alpha}} \wedge y_{\alpha,n_{\alpha}}],
\end{equation}
with $\lambda_{\alpha} \in K$ and $x_{\alpha,j}, y_{\alpha,j} \in \mf{g}$ such that 
\begin{equation} \label{WR4}
\sum_{\alpha} \lambda_{\alpha} [ [x_{\alpha,1},y_{\alpha,1}], \ldots, [x_{\alpha,n_{\alpha}},y_{\alpha,n_{\alpha}}]]=0
\end{equation}
in $\mf{g}$. Consider the homomorphism $\theta: \mf{g} \wedge \mf{g} \to \chi(\mf{g})/R$ defined by $\theta(x \wedge y) = [x,y^{\psi}]$. It is 
not hard to see that $\theta$ is well defined. Moreover, it follows by
\eqref{WR1}, \eqref{WR2}, \eqref{WR3} and \eqref{WR4} that $\theta$ induces a surjective homomorphism 
\[ \theta_1: ker(\phi) \to W/R,\]
thus $W/R$ is a quotient of $H_2(\mf{g};K)$.

\subsection{\texorpdfstring{$H_2(\mf{g};K)$}{H2(g;K)} is a quotient of \texorpdfstring{$W/R$}{W/R}}
Now we adapt the arguments in \cite{Sidki}, Section 4.1. Suppose that
\begin{equation} \label{stem}
 0 \to Z \to \mf{h} \to \mf{g} \to 0 
\end{equation}
is a stem extension of Lie algebras, that is, \eqref{stem} is an exact sequence, $Z$ is a central ideal of $\mf{h}$ and $Z \subseteq \mf{h}'$. 
Consider
\[ P = \langle \{(x,x,0), (0,x,x) \hbox{ } | \hbox{ } x \in \mf{h}\} \rangle \subseteq \mf{h} \oplus \mf{h} \oplus \mf{h}.\]
In other words, $P$ is the image of $\rho_{\mf{h}}: \chi(\mf{h}) \to \mf{h} \oplus \mf{h} \oplus \mf{h}$. It is not hard to check that $P$ can be described also as
\[ P = \{ (x,y,z) \in \mf{h} \oplus \mf{h} \oplus \mf{h} \hbox{ } | \hbox{ } x-y+z \in \mf{h}' \}.\]
Define
\begin{equation} \label{descrN}
 B = \{(z,z+z',z') \hbox{ } | \hbox{ } z,z' \in Z\} \subseteq P.
\end{equation}
Notice that $B$ is a central subalgebra of $P$, since $Z$ is central in $\mf{h}$. It follows that $P/B$ is a quotient of $\chi(\mf{h})$, by means of the homomorphism
$\theta: \chi(\mf{h}) \to P/B$ such that $\theta(x) = (x,x,0) + B$ and $\theta(x^{\psi}) = (0,x,x) + B$. Since $\theta(z) = \theta(z^{\psi})= 0$ for all 
$z \in Z$, it follows that $\theta$ factors through a homomorphism $\lambda: \chi(\mf{g}) \to P/B$, thus making the following diagram commutative:
\[ \xymatrix{ \chi(\mf{h}) \ar@{->>}[d] \ar@{->>}[r]^{\theta}  & P/B \\
\chi(\mf{g}) \ar@{->>}[ur]_{\lambda} & }
\]

\begin{lemma}  \label{lemma6.1}
We have:
\begin{enumerate}
 \item $R(\mf{g}) \subseteq ker(\lambda) \subseteq W(\mf{g})$,
 \item $\lambda(W(\mf{g})) \simeq Z$.
\end{enumerate}
\end{lemma}

\begin{proof}
 The first inclusion in (1) is clear, since $R(\mf{g})$ is the image of $R(\mf{h})$, and $R(\mf{h}) \subseteq ker(\theta)$. The second inclusion is also 
 clear, since
 $\rho (\chi(\mf{g})) \simeq P / (Z \oplus Z \oplus Z)$, which clearly is a quotient of $P/B$, and $\rho$ can be written as the composite:
 \[\xymatrixcolsep{1.5pc} \xymatrix{ \chi(\mf{g}) \ar@{->>}[r]^{\lambda} \ar@/_1pc/[rr]_{\rho} & P/B \ar@{->>}[r]^{\pi} & \rho( \chi(\mf{g})),} \]
 where $\pi$ is the canonical projection.
 
 As to item (2), notice that since $\lambda(W(\mf{g})) \subseteq ker(\pi)$, every element of $\lambda(W(\mf{g}))$ is 
 of the form $(z_1,z_2,z_3) + B$, with $z_i \in Z$. Such elements are clearly equivalent to elements of the form $(0,z,0) + B$ in $P/B$ for some $z \in Z$.
 Conversely, any such element must be in the image of $\lambda$, and actually it must be the image of some element of $W(\mf{g})$, since it projects to $0$
 in $\rho(\chi(\mf{g}))$. 
 
 Thus $\lambda(W(\mf{g})) = \{(0,z,0) + B \hbox{ } | \hbox{ } z \in Z \} \subseteq P/B$. The homomorphism $Z \to \lambda(W(\mf{g}))$ that takes $z$ to 
 $(0,z,0)+B$ is clearly well defined and surjective, and it also injective, by the description of $B$ in \eqref{descrN}. Thus $\lambda(W(\mf{g})) \simeq Z$.
  \end{proof}

By the lemma above we see that $W/R = W(\mf{g})/R(\mf{g})$ has $Z$ as quotient for every $Z$ that occurs as the kernel of some stem extension of $\mf{g}$. Now we 
show that $H_2(\mf{g};K)$ is one of such kernels, from what follows that $H_2(\mf{g}; K)$ is a quotient of $W/R$.

\begin{lemma}  \label{lemma6.2}
 Any Lie algebra $\mf{g}$ fits into a stem extension written as
 \[ 0 \to H_2(\mf{g}; K) \to \mf{h} \to \mf{g} \to 0.\]
\end{lemma}

\begin{proof}
 Write $\mf{g} = \mf{f}/\mf{n}$, where $\mf{f}$ is a free Lie algebra. Then by the Hopf formula
 \[ H_2(\mf{g}; K) = \frac{[\mf{f},\mf{f}] \cap \mf{n}}{[\mf{f},\mf{n}]}.\] 
 Thus we can see $H_2(\mf{g}; K)$ as a subalgebra of the abelian Lie algebra $\mf{n}/[\mf{f},\mf{n}]$. It follows that $H_2(\mf{g}; K)$ admits a 
 complement, that is, there is some $\mf{a} \subseteq \mf{n}$, with $[\mf{f},\mf{n}] \subseteq \mf{a}$, such that 
 \begin{equation}  \label{complSchur}
  \frac{\mf{n}}{[\mf{f},\mf{n}]} \simeq H_2(\mf{g}; K) \oplus \frac{\mf{a}}{[\mf{f},\mf{n}]}.
 \end{equation}
 Notice that $[\mf{f},\mf{a}] \subseteq [\mf{f},\mf{n}] \subseteq \mf{a}$, so $\mf{a}$ is an ideal of $\mf{f}$. Consider the exact sequence:
 \begin{equation} \label{stemext}
  0 \to \mf{n}/\mf{a} \to \mf{f}/\mf{a} \to \mf{f}/\mf{n} \to 0.
 \end{equation}
 The choice of $\mf{a}$ implies that $\mf{n}/\mf{a} \simeq H_2(\mf{g};K)$. Since $[\mf{f},\mf{n}] \subseteq \mf{a}$, the extension is central. By the direct sum description 
 \eqref{complSchur}, any element of $\mf{n}$ is equivalent modulo $\mf{a}$ to some $w \in [\mf{f},\mf{f}]$, so $\mf{n}/\mf{a} \subseteq [\mf{f}/\mf{a},\mf{f}/\mf{a}]$.
 Thus \eqref{stemext} is a stem extension.
 \end{proof}

Let $\lambda_2: W(\mf{g})/R(\mf{g}) \to H_2(\mf{g};K)$ be the homomorphism induced by the homomorphism arising in Lemma \ref{lemma6.1} 
when we take $Z$ to be $H_2(\mf{g}; K)$. By thinking of $H_2(\mf{g}; K)$ given by the Hopf formula for a fixed presentation of $\mf{g}$, 
as in Lemma \ref{lemma6.2}, we can write explicit expressions for $\lambda_2$ and for the isomorphism $\alpha: H_2(\mf{g};K) \to ker(\phi)$ of \cite{Ellis}.
It is not hard to see then that the composition
\[ \xymatrix{ H_2(\mf{g};K) \ar[r]^{\alpha} & ker(\phi) \ar[r]^{\theta_1} & W(\mf{g})/R(\mf{g}) \ar[r]^{\lambda_2} & H_2(\mf{g};K)}\]
is the identity. So $W(\mf{g})/R(\mf{g}) \simeq H_2(\mf{g};K)$, as we wanted.
 \end{section}

\begin{section}{Examples}
In this section we consider some examples. We begin by observing that we can describe $Im(\rho)$ as
\begin{equation} \label{Imrho}
Im(\rho) = \{(x,y,z) \in \mf{g} \oplus \mf{g} \oplus \mf{g} \hbox{ } |\hbox{ } x-y+z \in \mf{g}'\}
\end{equation}
for any Lie algebra $\mf{g}$. We want to obtain information about $W(\mf{g})$ and $R(\mf{g})$ in some particular cases.

\subsection{Abelian Lie algebras}

Let $\mf{g}$ be a finite dimensional abelian Lie algebra and let $x_1, \ldots, x_n$ be a basis. Then $\chi(\mf{g})$ is generated by the symbols $x_1, \ldots, x_n$ and
$x_1^{\psi}, \ldots, x_n^{\psi}$ with defining relations given by:
\begin{enumerate}
 \item $[x_i, x_j] = 0$ for all $i>j$;
 \item $[x_i^{\psi}, x_j^{\psi}] = 0$ for all $i>j$;
 \item $[x_i, x_i^{\psi}]=0$ for all $i$;
 \item $[x_i, x_j^{\psi}] = [x_i^{\psi},x_j]$ for all $i>j$.
\end{enumerate}
Notice that the elements $[x_i, x_j^{\psi}]$ are central. Indeed
\[[x_i, x_j^{\psi}]= -\frac{1}{2}[x_i-x_i^{\psi}, x_j- x_j^{\psi}],\]
and by the Jacobi identity, together with the fact that $[D,L]=0$, we get that $[[x_i-x_i^{\psi}, x_j- x_j^{\psi}],x_k]=0$ for all $k$, and similarly for
$x_k^{\psi}$.

\begin{prop}
 If $\mf{g}$ is an abelian Lie algebra of dimension $n$, then $\chi(g)$ is a Lie algebra of dimension $2n+ \binom{n}{2}$. We also have that $W=D$
 is a central ideal of dimension $\binom{n}{2}$, with $\chi(\mf{g})/W \simeq \mf{g} \oplus \mf{g}$. Finally, $R=0$.
\end{prop}

\begin{proof}
 The remarks above the proposition imply that  $D$ is linearly generated by $[x_i,x_j^{\psi}]$, for $i>j$. Each of these elements is clearly in the kernel of
 $\rho$, that is, in $W$, so $D = W$. Now $D/R \simeq W/R \simeq H_2(\mf{g};K) \simeq \bigwedge^2(\mf{g})$, thus $dim(D) \geq \binom{n}{2}$. 
 But then the elements $[x_i,x_j^{\psi}]$ with $1 \leq  j < i \leq n$ must be linearly independent and $R = 0$. Moreover, it is clear by \eqref{Imrho} that 
 $Im(\rho) \simeq \mf{g} \oplus \mf{g}$.
\end{proof}

\subsection{Perfect Lie algebras}
Let $\mf{g}$ be \textit{perfect}, that is, $\mf{g} = \mf{g}'$. Notice that in this case 
$Im(\rho) = \mf{g} \oplus \mf{g} \oplus \mf{g}$. Moreover, $W$ is a central ideal of $\chi(\mf{g})$. In fact, for 
$x,y \in \mf{g} \subseteq \chi(\mf{g})$ and $w \in W$, we have
\[ [[x,y],w] = [[x,w],y] + [x,[y,w]] = [[x,w],y^{\psi}] + [x,[y^{\psi},w]] = [[x,y^{\psi}],w]=0.\]
The first and the third equalities are instances of the Jacobi identity; the second and the fourth are consequences of $[L,D]=0$. Thus $W$ commutes with 
$\mf{g}' = \mf{g} \subseteq \chi(\mf{g})$, and similarly with $\mf{g}^{\psi}$. In this case $R(\mf{g}) = [\mf{g}, [L, \mf{g}^{\psi}]] = 0$. Indeed:
\[ R(\mf{g}) = [\mf{g}, [L, \mf{g}^{\psi}]] =[\mf{g}', [L, \mf{g}^{\psi}]] \subseteq [\mf{g}, [\mf{g}, [L, \mf{g}^{\psi}]]] = 0, \]
since $[\mf{g}, [L, \mf{g}^{\psi}]] \subseteq W$.

We conclude that $\chi(\mf{g})$ is a central (in fact stem) extension
of $H_2(\mf{g}; K)$ by $\mf{g} \oplus \mf{g} \oplus \mf{g}$. In particular, if $\mf{g}$ is \textit{superperfect}, that is, 
$\mf{g}$ is perfect and $H_2(\mf{g};K) = 0$, then $\chi(\mf{g}) \simeq \mf{g} \oplus \mf{g} \oplus \mf{g}$.

\subsection{Lie algebras generated by two elements}
We will show that $R(\mf{f})=0$ if $\mf{f}$ is free of rank $2$.

 \begin{rem}  \label{rem72}
 We will use repeatedly that for any $d \in D$ and $[x_1, \ldots, x_i, \ldots, x_n]$ an arbitrarily 
 arranged bracket of elements $x_i \in \mf{f} \cup \mf{f}^{\psi}$, we have 
 \[[[x_1, \ldots, x_i, \ldots, x_n], d]= [[x_1, \ldots, x_i^{\psi}, \ldots, x_n],d]\]
 for any $i$, as a consequence of $[D,L]=0$. If $x_i$ is already an element of $\mf{f}^{\psi}$, we interpret $\psi$ as an automorphism of order $2$,
 that is, $x_i = y^{\psi} \in \mf{f}^{\psi}$ and $x_i^{\psi} = (y^{\psi})^{\psi}=y \in \mf{f}$.
 \end{rem}

Let $\{x,y\}$ be a free basis of $\mf{f}$ and let $M$ be the set of monomials in these generators.
We want to show that $R(\mf{f}) = [\mf{f},[L,\mf{f}^{\psi}]]=0$. Clearly it is enough to show that 
\begin{equation} \label{eqR}
 \mc{R}(g,\ell,h) := [g,[\ell, h^{\psi}]]=0
\end{equation}
for all $g,h \in \mf{f}$ and $\ell \in L$. By linearity, it suffices to consider $g,h \in M$. We will show that actually it is enough to 
consider indecomposable monomials, that is, $g,h \in \{x,y\}$. For this, it suffices to show that if $g$ or $h$ can be written as a 
non-trivial bracket, then \eqref{eqR} follows from the identities with respect to each of the terms of the bracket.

First, if we have 
\[[[g_1,g_2],[\ell, h^{\psi}]] =  [g_1, [g_2, [\ell,h^{\psi}]]]-[g_2,[g_1,[\ell,h^{\psi}]]].\]
So if $\mc{R}(g_i, \ell, h)=0$ for $i=1,2$, then $\mc{R}([g_1,g_2],\ell, h)=0$ as well.

Similarly, suppose $h=[h_1,h_2]$. By the Jacobi identity we have:
\[ [g,[\ell, [h_1,h_2]^{\psi}]]  = [[\ell,h_1^{\psi}],[g,h_2^{\psi}]] + [[g,[\ell,h_1^{\psi}]],h_2^{\psi}]-[[g, [\ell, h_2^{\psi}]], h_1^{\psi}]- [[\ell,h_2^{\psi}],[g,h_1^{\psi}]]. \]
The first and the fourth terms in the right-hand side of the equation above vanish because $[D,L]=0$. The second and the third terms vanish if
we assume that $\mc{R}(g,\ell,h_1) = \mc{R}(g,\ell,h_2)=0$.

Now we want to do the same with respect to $\ell$. We will show that
is enough to consider the elements $\ell = m-m^{\psi}$, with $m \in \{x,y\}$.

Clearly it is enough to let $\ell$ run through a linear spanning set for $L$. We know that $L$ is generated as an algebra by the elements 
$m-m^{\psi}$ with $m \in M$.
Thus a spanning set for $L$ can be obtained by considering the long brackets involving these elements. 

Now recall that given $m,n,p \in M$, we have:
\[ [m-m^{\psi},[n-n^{\psi},p-p^{\psi}]] = [m,[n,p]]-[m,[n,p]]^{\psi}.\]
This was hinted in Remark \ref{mnu}. From this follows that $L$ is linearly spanned 
by elements of the form $m-m^{\psi}$ and $[m-m^{\psi},n-n^{\psi}]$ with $m,n \in M$. Now:
\[  [g,[[m-m^{\psi},n-n^{\psi}],h^{\psi}]] = [g,[[m-m^{\psi},h^{\psi}],n-n^{\psi}]] +[g , [m-m^{\psi}, [n-n^{\psi},h^{\psi}]]] = (\ast)\]
If $\mc{R}(n,m-m^{\psi},h)=\mc{R}(m,n-n^{\psi},h)=0$, then: 
\[(\ast)  = -[g,[[m-m^{\psi},h^{\psi}],n^{\psi}]] -[g , [m^{\psi}, [n-n^{\psi},h^{\psi}]]]\]
Then by the Jacobi identity, together with $[D,L]=0$, we have: 
\[(\ast)  = -[[g,[m-m^{\psi},h^{\psi}]],n^{\psi}] -[m^{\psi}, [g, [n-n^{\psi},h^{\psi}]]],\]
so $\mc{R}(g,[m-m^{\psi},n-n^{\psi}],h)=0$ if we also assume that $\mc{R}(g,m-m^{\psi},h)=0$ and $\mc{R}(g,n-n^{\psi},h)=0$.

We are down to: if $\mc{R}(g,m-m^{\psi},h)=0$ for all $g,h \in \{x,y\}$ and $m \in M$, then $R(\mf{f})=0$.

 Finally, it is enough to consider $m \in \{x,y\}$. In fact, if $m=[u,v]$, then 
  \[[g,[[u,v],h^{\psi}]]  =  [[g,[u,h^{\psi}]],v] + [[u, h^{\psi}], [g,v]]+ [[g,u],[v,h^{\psi}]] + [u, [g,[v,h^{\psi}]]].
 \]
 To see that $\mc{R}(g,[u,v]-[u,v]^{\psi},h)=0$, we need to show that we can change any instance of $u$ (resp. $v$) for $u^{\psi}$
 (resp. $v^{\psi}$) in the right-hand side of the equation above. For the first term 
 we use that $\mc{R}(g,u-u^{\psi},h)=0$ and then Remark \ref{rem72} (to change $v$ for $v^{\psi}$). The fourth term is analogous, but we 
 use that $\mc{R}(g,v-v^{\psi},h)=0$. For the 
 second and third terms we apply Remark \ref{rem72} twice.

By the arguments above, for $R(\mf{f})=0$, it is enough that $\mc{R}(g,m-m^{\psi},h)=0$ with $g,h,m \in \{x,y\}$.
But this can verified directly, being consequence of the relations $[x,x^{\psi}]=0$, $[y,y^{\psi}]=0$ and $[D,L]=0$. Thus:

\begin{prop}
 If $\mf{g}$ can be generated by two elements, then $R(\mf{g})=0$.
\end{prop}

\begin{proof}
 We have proved for $\mf{f}$ free of rank two. In general, if $\mf{g}$ is generated by two elements, then there is a surjective homomorphism
 $\varphi: \mf{f} \to \mf{g}$, which induces $\varphi_{\ast}: \chi(\mf{f}) \to \chi(\mf{g})$. It is clear that $\varphi_{\ast} (R(\mf{f})) = R(\mf{g})$,
 so the result follows.
\end{proof}

\begin{rem}
Observe that the proof above does not work for a free Lie algebra of rank greater than $2$, since we can not guarantee the base step: 
$[x^{\psi},[y-y^{\psi},z]]$ will not be trivial if $x$, $y$
and $z$ are three independent generators.
\end{rem}

\subsection{Other small Lie algebras}
Observe that for any finite dimensional Lie algebra $\mf{g}$, the following holds:
\begin{equation} \label{bounddim}
dim (\chi(\mf{g})) \geq 2 dim(\mf{g})+ dim(\mf{g}') + dim (H_2(\mf{g};K)). 
\end{equation}
Indeed, it follows by \eqref{Imrho} that
$dim(Im(\rho))= 2 dim(\mf{g})+ dim(\mf{g}')$. On the other hand, $dim(W) \geq dim(W/R) = dim (H_2(\mf{g};K))$, which gives the bound.
It is clear that \eqref{bounddim} is an equality if and only if $R(\mf{g})$ is trivial. By computing with GAP \cite{GAP4}, we were able to find Lie 
algebras $\mf{g}$ with coefficients in $\mathbb{Q}$ for which that does not happen.
These are:
\begin{enumerate}
 \item Let $\mf{g}$ be the Lie algebra generated by three elements $a$, $b$ and $c$ such that $[a,b]=[b,c]=[a,c]$ is a non-zero central element.
       Then:
       \[ H_2(\mf{g}; \mathbb{Q}) \simeq \mathbb{Q}^4, \hbox{  } dim(\chi(\mf{g}))=14,\hbox{  } dim (R(\mf{g}))= 1;\]
 \item If $\mf{g}$ is the free nilpotent Lie algebra of rank $3$ and class $2$, then: 
 \[ H_2(\mf{g}; \mathbb{Q}) \simeq \mathbb{Q}^8, \hbox{  } dim(\chi(\mf{g}))=27,\hbox{  } dim (R(\mf{g}))= 4.\]
 \end{enumerate}
 
We have the following corollary.

\begin{crlr}
If $\mf{f}$ is a free Lie algebra of rank at least $3$, then $R(\mf{f}) \neq 0$. 
\end{crlr}

\end{section}

\begin{section}{Remarks about the characteristic 2 case}
We will show that the conclusion of Lemma \ref{Lgenalg}, which is essential for the development of this results of this paper, fails almost always in characteristic $2$. 

\begin{prop}
 Suppose that $char(K)=2$. Then the ideal $L$ is generated by $\{x-x^{\psi} \hbox{ }| \hbox{ } x \in \mf{g} \}$ as a subalgebra if and only 
 if $\mf{g}$ is abelian and $\chi(\mf{g}) \simeq \mf{g} \oplus \mf{g}$.
\end{prop}

\begin{proof}
 The ``if'' direction is clear. Denote by $A$ the subalgebra of $\chi(\mf{g})$ generated by the elements $x-x^{\psi}$ and suppose $L = A$.

 First notice that $\rho(A)$ is generated by the elements $\rho(x-x^{\psi}) = (x,0,-x)$ for $x \in \mf{g}$. In characteristic $2$ the set of this elements is closed 
 by the bracket, as well as sum and multiplication by scalar:
 \[[(x,0,-x),(y,0,-y)] = ([x,y],0,[x,y]) = ([x,y],0,-[x,y]).\]
 It follows that $\rho(A) = \{(x,0,x) \hbox{ } | \hbox{ } x \in \mf{g}\}$. Now let $x, y \in \mf{g}$. Then $[x,y-y^{\psi}] \in L = A$. But
 \[\rho([x,y-y^{\psi}]) = [(x,x,0), (y,0,-y)] = ([x,y],0,0) \in \rho(A),\]
 therefore $[x,y]= 0$, that is, $\mf{g}$ is abelian.
 
 Now define $\sigma: \chi(\mf{g}) \to \mf{g}$ by $\sigma(x)=x$ and $\sigma(x^{\psi})=0$. It is clear that $\sigma(A) = \mf{g}$. 
 Notice that $[x,y^{\psi}] = [x^{\psi},y]$ still holds. Then:
 \[[x-x^{\psi},y-y^{\psi}] = [x,y] - [x,y^{\psi}] - [x^{\psi},y] + [x,y]^{\psi} = 2[x,y^{\psi}] = 0.\]  
 The inverse for $\sigma |_A$ is then well defined. Now for any 
 $x,y \in \mf{g}$, we have $[x,y^{\psi}] = [x,y-y^{\psi}] \in L = A$. But $\sigma([x,y^{\psi}]) = 0$, so $[x,y^{\psi}]=0$. Thus $[\mf{g}, \mf{g}^{\psi}]=0$
 in $\chi(\mf{g})$, as we wanted.
\end{proof}
\end{section}

\section*{Acknowledgements}
The author would like to thank D. H. Kochloukova and C. Martínez-Pérez for the guidance of this work. He also thanks the referee for the nice suggestions.
The author is supported by grants 2015/22064-6 and 2016/24778-9, from S\~{a}o Paulo Research Foundation (FAPESP).

\end{document}